\documentclass{article}
\usepackage[a4paper, margin=2.5cm]{geometry} 
\usepackage[medium]{titlesec}

\usepackage{amssymb} 
\usepackage{amsmath}
\usepackage{amsthm}

\makeatletter
\newcommand*\owedge{\mathpalette\@owedge\relax}
\newcommand*\@owedge[1]{%
  \mathbin{%
    \ooalign{%
      $#1\m@th\bigcirc$\cr
      \hidewidth$#1\m@th\wedge$\hidewidth\cr
    }%
  }%
}
\makeatother

\usepackage[utf8]{inputenc} 
\usepackage[T1]{fontenc} 
\usepackage{lmodern}

\usepackage{bbm}
\usepackage{hyperref}
\usepackage[blocks]{authblk}

\newcommand{\R}{{\mathbb R}}

\newcommand{\rd}{{\mathbbm d}}

\newtheorem{lm}{Lemma}

\newtheorem{prop}[lm]{Proposition}
\newtheorem*{prop*}{Proposition}
\newtheorem{cor}[lm]{Corollary}
\theoremstyle{remark}

\theoremstyle{remark}

\title{Spiraling conformal geodesics}

\author[1]{Wojciech Kami{\'n}ski}
\affil[1]{Faculty of Physics, University of Warsaw,
ul. Pasteura 5, 02-093 Warsaw, Poland}

\begin{document}

\maketitle

\begin{abstract}
In this short note, we construct an example of spiraling conformal geodesic in Euclidean signature in dimension $3$, answering the question posed by Helmuth Friedrich and Paul Tod, if such objects exists. Our example is not real analytic, but similar constructions can lead also to real analytic metrics in arbitrary dimensions.\footnote{We thank Paul Tod and Maciej Dunajski for this observation.}
\end{abstract}

\section{Introduction}

Conformal geodesics \cite{yano2020theory} (also called conformal circles) are special curves that generalizes concept of geodesics from  pseudo-Riemannian geometry to the realm of conformal manifolds. A conformal geodesic for one metric remains conformal geodesic also after conformal transformation. These curves play an important role in general relativity as a tool to construct gaussian type coordinate system which behave nicely under conformal transformations \cite{schmidt1974new}, that is crucial at conformal boundaries of spacetimes.

Due to their importance in the study of initial-boundary value problem for general relativity \cite{schmidt1974new, friedrich1987conformal} and the character of distinguished curves for conformal geometry \cite{bailey1990conformal, bailey1994thomas, gover2018distinguished}, this class of curves attracted some attention. It turned out that despite some similarities, many nice properties of standard geodesics do not hold for conformal geodesics. For example, the variational principle exists only in dimension $3$ \cite{kruglikov2024variationality, kruglikov2025conformal} unless one resort to variational principles with some constraints \cite{dunajski2021variational}. There exists a version of gaussian neighbourhood (so-called heart) introduced in \cite{cameron2024conformal}, but it has a bit weaker properties than for metric geodesics. There are infinitely many conformal geodesics connecting two neighbouring points and there is no notion of `minimal length'.

Most importantly, unlike geodesics, conformal circles can develop various singularities.
Based on many examples, Paul Tod conjectured \cite{tod2012some} that one type of such singularities can be ruled out: that conformal geodesics cannot spiral. Namely, that they cannot converge to a point without reaching it in a finite proper time (we will provide explicite definition later). The conjecture was shown in the case of Einstein metrics \cite{tod2012some}. Other examples considered in \cite{tod2012some} were highly symmetric, but in \cite{dunajski2022conformal} more general manifolds were considered supporting the claim. Additionally, existence of hearts \cite{cameron2024conformal} suggested that spiraling cannot occure (however, see \cite{cameron2024conformal-erratum} why method from \cite{cameron2024conformal} does not exclude spiraling conformal geodesics).

In this paper, we provide an explicit example of a $3$-dimensional Riemannian metric with a spiraling conformal geodesic. Although the example is not real-analytic, it can be modified slightly to yield a real-analytic metric with the same property (private communication by Paul Tod). In this way, the original conjecture is resolved. 

Let us now describe how our example is constructed. In flat Euclidean space, every circle is a conformal geodesic. This looks like a very unstable situation: under small perturbation of the conformal geodesic equation one can achieve that the radius of the circle shrinks slowly. Since every circle is a solution of the equation, the error can be very small if the changes of the radius becomes slower and slower as we approach the origin. The crucial point is that we want these corrections to result from a perturbation of the metric, due to conformal geodesics equation. Here comes the second insight. We can consider a totally geodesic hypersurface with the flat induced metric. For a curve contained in this hypersurface the conformal geodesic equation is the same as the equation on the hypersurface except one term involving the Schouten tensor, which is now the tensor for the bigger space. This observation allows us to utilize the additional dimension to create non-trivial desired correction to the flat metric equation. To make things simple, our hypersurface is $\{z=0\}\subset \R^3$ and we start with proposing a curve (Section \ref{sec:spirals}) and computing corrections which should be implemented in order to make this curve an unparametrized conformal geodesic. For this purpose, we use an alternative equation for unparametrized conformal geodesics (Lemma \ref{lm:uparam}). Then, we propose a metric in $\R^3$ which implements these corrections by a proper form of the Schouten tensor (Section \ref{sec:metric}). 

Our curve approaches the origin very slowly in terms of the proper time. This has the advantage of simplifying certain conditions for the smoothness of the metric, but the drawback is that the resulting metric is not real-analytic. By modifying the curve to a spiral that approaches the origin slightly faster—though not too fast—one can obtain a real-analytic metric.\footnote{Private communication by Paul Tod.}

\section{The example}

%This note gives an example of a spiral conformal geodesic in Riemannian signature. 

We are working always in Riemannian signature, $|\cdot|$ denotes the length of a vector, $\nabla$ is the Levi-Civita covariant derivative and  $L_{\mu\nu}$ is the Schouten tensor.
We will now define main objects:
\begin{enumerate}
\item A {\it spiral} to point $p$ is a curve $\R_+\ni s\rightarrow \gamma(s)$ parametrized by a proper time parameter (velocity has norm one) such that for every open neighbourhood $U$ of $p$
\begin{equation}
\exists s_0\in \R_+\ \forall s>s_0\colon \gamma(s)\in U.
\end{equation}
\item For every two co-vector $H$ we denote $(\hat{H}\vec{u})^\mu=H^{\mu\nu}u_\nu$. A (parametrized by proper time) {\it conformal geodesic} in Riemannian signature is a curve satisfying an equation
\begin{equation}
\nabla_{\vec{u}}\vec{a}=\vec{u}\left(-|\vec{a}|^2-\vec{u}\cdot \hat{L}\vec{u}\right)+\hat{L}\vec{u},
\end{equation}
where $\vec{a}=\nabla_{\vec{u}}\vec{u}$, $\vec{u}$ is the velocity ($|\vec{u}|^2=1$) and $L_{\mu\nu}$ is the Schouten tensor (thus $\hat{L}$ is defined as above).  Importantly, if $I\ni s\rightarrow \gamma(s)$ is a conformal geodesic in proper time parametrization then such is also $I\ni s\rightarrow \gamma(-s)$. 
\item An unparametrized conformal geodesic is a curve $I\ni t \rightarrow \gamma(t)$ such that there exists a smooth reparametrization after which the curve becomes a conformal geodesic parametrized by the proper time.
\end{enumerate}

\subsection{Conformal geodesics}

Although the example can be checked directly, it is simpler to use some equivalent formulation of equations of conformal geodesics.

\begin{lm}
Conformal geodesics equation is equivalent to condition
\begin{equation}\label{eq:wedge-1}
\nabla_{\vec{u}}(\vec{u}\wedge \vec{a})=\vec{u}\wedge \hat{L}\vec{u},\quad |\vec{u}|^2=1.
\end{equation}
\end{lm}

\begin{proof}
Let us notice that
\begin{equation}
\nabla_{\vec{u}}(\vec{u}\wedge \vec{a})=\vec{a}\wedge \vec{a}+\vec{u}\wedge\nabla_{\vec{u}}\vec{a}=\vec{u}\wedge\nabla_{\vec{u}}\vec{a}.
\end{equation}
Thus, the equation \eqref{eq:wedge-1} is equivalent to
\begin{equation}
\nabla_{\vec{u}}\vec{a}=c\vec{u}+\hat{L}\vec{u},
\end{equation}
where $c(t)$ is some function.
Now, from constancy of norm of $\vec{u}$ we have $\vec{u}\cdot\vec{a}=0$,
\begin{equation}
0=\nabla_{\vec{u}}(\vec{u}\cdot\vec{a})=|\vec{a}|^2+\vec{u}\cdot \nabla_{\vec{u}}\vec{a}=|\vec{a}|^2+c+\vec{u}\cdot \hat{L}\vec{u},
\end{equation}
so $c=-|\vec{a}|^2-\vec{u}\cdot \hat{L}\vec{u}$ and the equation is equivalent to equation of conformal geodesics.
\end{proof}

We can recast the condition for unparametrized curves.

\begin{lm}\label{lm:uparam}
Unparametrized conformal geodesics equation is equivalent to
\begin{equation}
\nabla_{\vec{v}}\frac{\vec{v}\wedge \vec{b}}{|\vec{v}|^3}=\frac{\vec{v}\wedge \hat{L}\vec{v}}{|\vec{v}|},
\end{equation}
where $\vec{v}$ is velocity and $\vec{b}=\nabla_{\vec{v}}\vec{v}$.
\end{lm}

\begin{proof}
Let us denote by $s$ the proper time and $|\vec{v}|=\frac{ds}{dt}$
\begin{equation}
\vec{v}=|\vec{v}|\vec{u},\quad \vec{b}=|\vec{v}|\left(\frac{d|\vec{v}|}{ds}\vec{u}+|\vec{v}|\vec{a}\right).
\end{equation}
The formulas simplify in the case of wedge product
\begin{equation}
\vec{v}\wedge \vec{b}=|\vec{v}|^3\vec{u}\wedge \vec{a}.
\end{equation}
Summing up
\begin{equation}
\nabla_{\vec{u}}(\vec{u}\wedge \vec{a})=|\vec{v}|^{-1}\nabla_{\vec{v}}\frac{\vec{v}\wedge \vec{b}}{|\vec{v}|^3}.
\end{equation}
Additionally, $\vec{v}\wedge \hat{L}\vec{v}=|\vec{v}|^2\vec{u}\wedge \hat{L}\vec{u}$, thus the result.
\end{proof}

\subsection{Spirals}\label{sec:spirals}

Consider two dimensional Euclidean space and polar coordinates $(r,\phi)$,
\begin{equation}
\rd s^2=\rd r^2+r^2\rd \phi^2.
\end{equation}
We introduce orthonormal basis of co-vectors $\check{e}_r=\rd r$, $\check{e}_\phi=r\rd \phi$ and the dual basis of vectors $\hat{e}_r, \hat{e}_\phi$.

We introduce some objects that will play a role in our example:
\begin{enumerate}
\item A symmetric two tensor (two co-vector)
\begin{equation}
M=2\check{e}_r\check{e}_\phi=2r\rd r\rd \phi.
\end{equation}
We remark that this object, in contrary to the volume form $r\rd r\wedge \rd \phi$, is not smooth at the origin.
\item A curve for $t\in (0,1]$ given in the polar coordinates by 
\begin{equation}\label{eq:curve}
r(t)=t,\ \phi(t)=e^{\frac{1}{t}}.
\end{equation}
One can show that after proper time reparametrization, this curve is a spiral that converges to the origin $r=0$ for $t\rightarrow 0_+$.
\end{enumerate}
For a function $f(t)$ on $\R_+$ we will write $f(t)=O(t^\infty)$ iff 
\begin{enumerate}
\item this function is smooth
\item it extends smoothly by zero to the negative numbers. Namely, the limits of all derivatives vanishes at $0$.
\end{enumerate}
Importantly, if $f(t)=O(t^\infty)$ then also $t^{-1}f(t)=O(t^\infty)$ and $\frac{df}{dt}=O(t^\infty)$.

\begin{lm}\label{lm:curve}
There exists a smooth function $k(t)=O(t^\infty)$  such that for the curve \eqref{eq:curve} in Euclidean space
\begin{equation}
\nabla_{\vec{v}}\frac{\vec{v}\wedge \vec{b}}{|\vec{v}|^3}=k(t) \frac{\vec{v}\wedge \hat{M}\vec{v}}{|\vec{v}|}.
\end{equation}
Moreover, the length of the curve is infinite $\int_0^1dt\ |\vec{v}|=\infty$.
\end{lm}

\begin{proof}
Let us denote $f(t)=\frac{1}{t^2}e^{\frac{1}{t}}$. Importantly $\frac{1}{f}=O(t^\infty)$ so by differentiating $\frac{\dot{f}}{f^2}=O(t^\infty)$. Additionally for any $n$, $\frac{1}{t^nf}=O(t^\infty)$. 

We use formulas for velocity and acceleration in polar coordinates
\begin{equation}
\vec{v}=\dot{r}\hat{e}_r+r\dot{\phi}\hat{e}_\phi,\quad \vec{b}=\left(\ddot{r}-r \dot{\phi}^2\right) \hat{e}_r+(r \ddot{\phi}+2 \dot{r} \dot{\phi}) \hat{e}_\phi.
\end{equation}
Let us notice that
\begin{equation}
r=t,\ \dot{r}=1,\ \ddot{r}=0,\quad \dot{\phi}=-f,\ \ddot{\phi}=-\dot{f}.
\end{equation}
This allows us to write
\begin{equation}
\vec{v}=-tf(\hat{e}_\phi+O(t^\infty)\hat{e}_r),\quad |\vec{v}|=tf(1+O(t^\infty)).
\end{equation}
Similarly for the acceleration the dominant term comes from the centrifugal part
\begin{equation}
\vec{b}=-tf^2(\hat{e}_r+O(t^\infty)\hat{e}_r+O(t^\infty)\hat{e}_\phi).
\end{equation}
We introduce a natural invariant preserved by covariant derivative
\begin{equation}
\mu:=\hat{e}_r\wedge \hat{e}_\phi.
\end{equation}
We can now compute
\begin{equation}
\vec{v}\wedge \vec{b}=-t^2f^3(\hat{e}_r\wedge \hat{e}_\phi+O(t^\infty)\hat{e}_r\wedge \hat{e}_\phi)=-t^2f^3(1+O(t^\infty))\mu.
\end{equation}
Together with the formula for the norm of velocity and covariant constancy of $\mu$
\begin{equation}\label{eq:eq1}
\nabla_{\vec{v}}\frac{\vec{v}\wedge \vec{b}}{|\vec{v}|^3}=\left(\frac{d}{dt}\frac{-t^2f^3(1+O(t^\infty))}{t^3f^3(1+O(t^\infty))}\right)\mu=\left(\frac{d}{dt}\left(-\frac{1}{t}+O(t^\infty)\right)\right)\mu=\left(\frac{1}{t^2}+O(t^\infty)\right)\mu.
\end{equation}
Similarly,
\begin{equation}
\hat{M}\vec{v}=-tf(\hat{e}_r+O(t^\infty)\hat{e}_\phi)\Longrightarrow \vec{v}\wedge \hat{M}\vec{v}=-t^2f^2(1+O(t^\infty))\mu,
\end{equation}
so
\begin{equation}\label{eq:eq2}
\frac{\vec{v}\wedge \hat{M}\vec{v}}{|\vec{v}|}=-tf(1+O(t^\infty))\mu.
\end{equation}
Comparing both expressions \eqref{eq:eq1} and \eqref{eq:eq2} we see that $k(t)=-\frac{1}{t^3f}(1+O(t^\infty))=O(t^\infty)$.

Finally $|\vec{v}|>tf$ thus as $e^{\frac{1}{t}}\geq 1$ for $t>0$
\begin{equation}
\int_0^1dt\ |\vec{v}|>\int_0^1dt\ tf\geq \int_0^1dt\ \frac{1}{t}=\infty.
\end{equation}
This shows that the length is infinite.
\end{proof}

\begin{cor}\label{cor:spiral}
After proper time reparametrization, the curve \eqref{eq:curve} is a spiral that converges to the origin $r=0$ for $t\rightarrow 0_+$.
\end{cor}

\subsection{A spiraling conformal geodesic}\label{sec:metric}

We will now define a metric for which our spiral is a conformal geodesic. The idea is to assume that $\{z=0\}$ is a totally geodesic surface with flat metric, but the higher order terms in $z$ allow us to introduce non-trivial Schouten tensor which will play a role of $M(t)$ in our example. Terms of order $O(z^4)$ are introduced only to make our metric well-define on the whole $\R^3$.

\begin{lm}\label{lm:Ricci}
For every function $h(r)=O(r^\infty)$,  a metric on $\R^3$ given in cylindrical coordinates $(x,y,z)=(r\cos\phi,r\sin\phi,z)$ by
\begin{equation}
\rd s^2=\left(1+h(r)^2z^4\right)\left(\rd r^2+r^2\rd \phi^2\right)+4h(r)z^2r\rd r\rd \phi+\rd z^2
\end{equation}
is a smooth Riemannian metric. Moreover,
\begin{enumerate}
\item Metric restricted to $\Sigma=\{z=0\}$ is $\rd x^2+\rd y^2=\rd r^2+r^2\rd \phi^2$ and extrinsic curvature of $\Sigma$ vanishes,
\item The Ricci tensor $R$ at $z=0$ satisfies
\begin{equation}
R|_{\Sigma}=-4h(r)r\rd r\rd \phi=-2h(r)M.
\end{equation}
\end{enumerate}
\end{lm}

\begin{proof}
Indeed,
\begin{equation}
r^3\rd r\rd \phi=(x\rd x+y\rd y)(x\rd y-y\rd x),\quad \rd r^2+r^2\rd \phi^2=\rd x^2+\rd y^2
\end{equation}
are smooth. Moreover, 
\begin{equation}
\frac{h(r)}{r^2}, \quad h(r)^2
\end{equation}
are smooth functions on $\R^2$ because they vanishes to all orders at the origin which would be the only troublesome point. Thus,
\begin{equation}
z^2hr\rd r\rd \phi=z^2\frac{h}{r^2}r^3\rd r\rd \phi,\quad (1+h^2z^4)\left(\rd r^2+r^2\rd \phi^2\right)
\end{equation}
are smooth and the metric is smooth.

The metric can be written as
\begin{equation}
\rd s^2=\left(\check{e}_r+hz^2\check{e}_\phi\right)^2+\left(\check{e}_\phi+hz^2\check{e}_r\right)^2+\rd z^2,
\end{equation}
thus it is Riemannian for $r\not=0$. At $r=0$ the metric is $\rd x^2+\rd y^2+\rd z^2$ in Cartesian coordinates thus it is also non-degenerate.

Let us notice that
\begin{equation}
\rd s^2=\rd x^2+\rd y^2+\rd z^2+2z^2 h(r)M+O(z^3),
\end{equation}
thus the claim about the induced metric and extrinsic curvature. The Riemann tensor can be directly computed as in the Cartesian coordinates the first derivatives of the metric vanish at $\Sigma$.
\begin{equation}
Riem=-2\rd z^2 \owedge h(r)M,
\end{equation}
where $\owedge$ is Kulkarni-Nomizu product. In particular as trace of $M$ vanishes and every contraction of $M$ with $\rd z^2$ also vanish we obtain the following formula for Ricci
\begin{equation}
\hat{R}=-2h(r)\hat{M}.
\end{equation}
That finishes the proof.
\end{proof}

\begin{prop}
Consider a smooth metric on $\R^3$ such that in cylindrical coordinates $(x,y,z)=(r\cos\phi,r\sin\phi,z)$
\begin{equation}
\rd s^2=\left(1+h(r)^2z^4\right)\left(\rd r^2+r^2\rd \phi^2\right)+4h(r)z^2r\rd r\rd \phi+\rd z^2
\end{equation}
where $h(r)=-\frac{1}{2}k(r)$ and $k(r)$ is the function from Lemma
\ref{lm:curve}. Then the curve \eqref{eq:curve} (in plane $z=0$) is a spiraling conformal geodesic (after proper time reparametrization).
\end{prop}

\begin{proof}
Let us remind that by Lemma \ref{lm:Ricci} the metric is smooth, the plane $z=0$ is totally geodesic and its induced metric is flat. Thus, by Lemma \ref{lm:curve}
\begin{equation}
\nabla_{\vec{v}}\frac{\vec{v}\wedge \vec{b}}{|\vec{v}|^3}=k(t) \frac{\vec{v}\wedge \hat{M}\vec{v}}{|\vec{v}|}.
\end{equation}
Let us notice that Ricci and Schouten tensors differ in dimension $3$ by a term proportional to the metric, thus
\begin{equation}
\vec{v}\wedge \hat{L}\vec{v}=\vec{v}\wedge \hat{R}\vec{v}.
\end{equation}
Moreover, by Lemma \ref{lm:Ricci}
\begin{equation}
\hat{R}\vec{v}=-2h(r)\hat{M}\vec{v}=k(r)\hat{M}\vec{v}
\end{equation}
Additionally, on the curve $r=t$ thus
\begin{equation}
k(t) \frac{\vec{v}\wedge \hat{M}\vec{v}}{|\vec{v}|}=\frac{\vec{v}\wedge k(r)\hat{M}\vec{v}}{|\vec{v}|}=\frac{\vec{v}\wedge \hat{R}\vec{v}}{|\vec{v}|}=\frac{\vec{v}\wedge \hat{L}\vec{v}}{|\vec{v}|}.
\end{equation}
and the curve is an unparametrized conformal geodesic by Lemma \ref{lm:uparam}. After proper time reparametrization, it is a spiral by Corollary \ref{cor:spiral}.
\end{proof}

\section{Summary and outlook}

We constructed an example of a spiraling conformal geodesic, proving that this kind of singular behavior is possible for these curves. However, there are several classes of metrics for which spirals cannot occur. The most important among them are Einstein metrics  \cite{tod2012some} and gravitational instantons \cite{dunajski2022conformal}. As there are both examples of metrics with spirals and metrics in which spirals are excluded, the natural question is whether spiraling conformal geodesics are generic and if not whether there is a effective criterium (for example in terms of non-vanishing of some conformal invariant) for excluding conformal spirals to a given point. We should stress that these questions can be also asked in other than Riemannian signatures (most importantly in Lorentzian signature) and in these cases even less is known.

\section*{Acknowledgements}

We would like to thank Peter Cameron, Maciej Dunajski and Paul Tod for comments and discussion about the issue of spiraling conformal geodesics.

\bibliographystyle{ieeetr}
%\bibliography{spirals}

\end{document}